\documentclass{amsart}
\usepackage[textsize=scriptsize]{todonotes}

\usepackage{etoolbox}
\usepackage[margin=1in,marginparwidth=0.8in]{geometry}
\usepackage{tikz-cd}
\usepackage{amsmath}
\usepackage{amssymb}
\usepackage{amsthm}
\usepackage{amscd}
\usepackage{enumerate}
\usepackage[pdfusetitle,unicode,hidelinks]{hyperref}
\usepackage{bbm}
\usepackage{etoolbox}

\usepackage[utf8]{inputenc}
\usepackage[T1]{fontenc}

\newtheorem{proposition}{Proposition}
\newtheorem{corollary}[proposition]{Corollary}
\newtheorem{lemma}[proposition]{Lemma}
\newtheorem{theorem}[proposition]{Theorem}

\newtheorem*{conjecture*}{Conjecture}
\newtheorem*{theorem*}{Theorem}
\newtheorem*{corollary*}{Corollary}
\newtheorem*{proposition*}{Proposition}
\newtheorem*{lemma*}{Lemma}
\theoremstyle{definition}

\newtheorem*{definition*}{Definition}
\newtheorem*{construction*}{Construction}
\theoremstyle{remark}
\newtheorem{remark}[proposition]{Remark}

\newtheorem*{remark*}{Remark}
\newtheorem*{variant*}{Variant}

\newtheorem*{example*}{Example}

\newcommand{\id}{\operatorname{id}}
\newcommand{\Z}{\mathbb{Z}}

\newcommand{\F}{\mathbb{F}}

\let\bb=\mathbb

\def\A{\bb A}

\let\lim=\relax
\DeclareMathOperator*{\lim}{lim}

\def\CAlg{\mathrm{CAlg}}

\def\mot{\mathrm{mot}}
\newcommand{\et}{{\acute{e}t}}

\newcommand{\comp}{{{\kern -.5pt}\otimes}}
\newcommand{\syn}{\mathrm{syn}}

\newcommand{\cycsp}{\mathrm{CycSp}}

\newcommand{\TC}{\mathrm{TC}}

\newcommand{\THH}{\mathrm{THH}}
\newcommand{\hh}{\mathrm{HH}}

\newcommand{\thh}{\THH}
\newcommand{\tp}{\mathrm{TP}}
\newcommand{\tc}{\TC}

\newcommand{\NN}{\ensuremath{\mathbb{N}}}
\newcommand{\ZZ}{\ensuremath{\mathbb{Z}}}

\newcommand{\FF}{\ensuremath{\mathbb{F}}}
\newcommand{\TT}{\ensuremath{\mathbb{T}}}

\usepackage{relsize}
\usepackage[bbgreekl]{mathbbol}
\usepackage{amsfonts}
\DeclareSymbolFontAlphabet{\mathbb}{AMSb} 
\DeclareSymbolFontAlphabet{\mathbbl}{bbold}
\newcommand{\Prism}{{\mathlarger{\mathbbl{\Delta}}}}

\numberwithin{proposition}{section}
\numberwithin{equation}{section}

\theoremstyle{proposition}

\title{On the $K$-theory of the $p$-adic unit disk}
\date{\today}

\author{Elden Elmanto}

\author{Noah Riggenbach}

\DeclareMathSymbol{A}{\mathalpha}{operators}{`A}
\DeclareMathSymbol{B}{\mathalpha}{operators}{`B}
\DeclareMathSymbol{C}{\mathalpha}{operators}{`C}
\DeclareMathSymbol{D}{\mathalpha}{operators}{`D}
\DeclareMathSymbol{E}{\mathalpha}{operators}{`E}
\DeclareMathSymbol{F}{\mathalpha}{operators}{`F}
\DeclareMathSymbol{G}{\mathalpha}{operators}{`G}
\DeclareMathSymbol{H}{\mathalpha}{operators}{`H}
\DeclareMathSymbol{I}{\mathalpha}{operators}{`I}
\DeclareMathSymbol{J}{\mathalpha}{operators}{`J}
\DeclareMathSymbol{K}{\mathalpha}{operators}{`K}
\DeclareMathSymbol{L}{\mathalpha}{operators}{`L}
\DeclareMathSymbol{M}{\mathalpha}{operators}{`M}
\DeclareMathSymbol{N}{\mathalpha}{operators}{`N}
\DeclareMathSymbol{O}{\mathalpha}{operators}{`O}
\DeclareMathSymbol{P}{\mathalpha}{operators}{`P}
\DeclareMathSymbol{Q}{\mathalpha}{operators}{`Q}
\DeclareMathSymbol{R}{\mathalpha}{operators}{`R}
\DeclareMathSymbol{S}{\mathalpha}{operators}{`S}
\DeclareMathSymbol{T}{\mathalpha}{operators}{`T}
\DeclareMathSymbol{U}{\mathalpha}{operators}{`U}
\DeclareMathSymbol{V}{\mathalpha}{operators}{`V}
\DeclareMathSymbol{W}{\mathalpha}{operators}{`W}
\DeclareMathSymbol{X}{\mathalpha}{operators}{`X}
\DeclareMathSymbol{Y}{\mathalpha}{operators}{`Y}
\DeclareMathSymbol{Z}{\mathalpha}{operators}{`Z}

\begin{document}

\maketitle

\begin{abstract} In this note, we study the $p$-complete topological cyclic homology of the affine line relative to a ring $A$ which is smooth over a perfectoid ring $R$. Denoting by $NTC(A; \Z_p)$ the spectrum which measures the failure of $\A^1$-invariance on $A$, we observe a kind of Quillen-Lichtenbaum phenomena for $NTC(A; \Z_p)$ --- that it is isomorphic to its own $K(1)$-localization in a specified range of degrees which depends on the relative dimension of $A$. Somewhat surprisingly, this range is better than considerations following from a theorem of Bhatt-Mathew and \'etale-to-syntomic comparisons. Via the Dundas-Goodwillie-McCarthy theorem, we obtain a description of the algebraic $K$-theory of $p$-completed affine line over such rings.

\end{abstract}

\tableofcontents
\section{Introduction} The goal of this note is to expand on a new strand of calculations in ($p$-adic) algebraic $K$-theory first discovered, at least to the authors' knowledge, in the work of the second author \cite{riggenbach2022cusps} where the $K$-theory of cuspidal curves were computed over perfectoid rings. Over these $p$-adic bases,  one can either consider the cusp as an algebraic scheme or as a formal scheme. The answers turn out to differ. The key difference arises from the following pheomena: let $R$ be a perfectoid ring, then Antieau-Mathew-Morrow observed that $K(R[t]) \simeq K(R)$ \cite{amm-perfd}, suggesting that $R$ behaves like a regular ring. On the other hand if we $p$-complete the affine line and set $R\langle t \rangle := R[t]^\wedge_p $, it is no longer the case that $K(R\langle t \rangle) \simeq K(R)$. The resulting calculation, at least after profinite completion, is recorded by the second author in \cite[Theorem 1.2]{riggenbach2022cusps}. 

An interesting feature of this calculation is that the relative groups $K(R\langle t \rangle, (t); \widehat{\Z})$ has vanishing even homotopy groups and that the odd homotopy groups are all isomorphic. In fact, the mechanism by which this happens is a ``Quillen-Lichtenbaum'' style phenomena for topological cyclic homology ($TC$) of the affine line. To explain this, we set
\[
NTC(R;\Z_p):= \mathrm{fib}(TC(R[t], \Z_p) \rightarrow TC(R; \Z_p));
\]
when we apply $K$-theory in place of $TC(-;\Z_p)$ we recover Bass' $NK$ spectra whose homotopy groups are the $K$-theory of nilpotent endomrophisms (up to a shift). Since $TC$ only depends on the input up to $p$-adic completion, we have that 
\[
TC(R[t];\Z_p) \simeq TC(R\langle t \rangle; \Z_p).
\] The henselian invariance results of \cite{clausen2018k} provides an equivalence
\[
\tau_{\geq 0}NTC(R;\Z_p) \simeq K(R\langle t \rangle, (t); \Z_p).
\] The key calculation in \cite[Theorem 4.14]{riggenbach2022cusps} witnesses $NTC(R;\Z_p)$ as the connective cover of its $K(1)$-localization which is evidently $2$-periodic. As explained in \emph{loc. cit.} the map from $NTC(R;\Z_p) \rightarrow L_{K(1)}NTC(R;\Z_p)$ is, in fact, $(-2)$-truncated\footnote{Recall that a map of spectra $X \rightarrow Y$ is said to be $n$-truncated if the fibre $F$ is $n$-truncated, that is, $\pi_i(F) = 0$ for $i > n$. In other words, the map $X \rightarrow Y$ induces an isomorphism on $\pi_i$ for $i \geq n+2$ and an injection in degree $i\geq n+1$.}. As we will explain in this paper, it is not that surprising that this map is $n$-truncated for some $n$ granting certain considerations from motivic filtrations; in fact the map $TC(R;\Z_p) \rightarrow L_{K(1)}TC(R;\Z_p)$ is $n$-truncated for some $n$. However, the specific bound on the $NTC$ part is somewhat surprising and feeds back information into the syntomic cohomology of $R\langle t \rangle$.

The first main result of this paper is a generalization of this Quillen-Lichtenbaum phenomena observed in \cite{riggenbach2022cusps}.

\begin{theorem}\label{thm:main}
Let $R$ be a perfectoid ring. Let $A$ be the $p$-adic completion of smooth $R$-algebra of relative dimension $d$. Then the map \[N\tc(A;\ZZ_p)\to L_{K(1)}N\tc(A;\ZZ_p)\] is $(d-1)$-truncated and an isomorphism \footnote{This is equivalent to saying that it is a $\tau_{\geq d}$-equivalence.} in degree $d$. If $R$ is further assumed to be $p$-torsion free then this map is also $(d-2)$-truncated. In particular, the map $K(A\langle t\rangle, (t); \ZZ_p) \rightarrow L_{K(1)}K(A\langle t\rangle, (t); \ZZ_p)$ is $(d-1)$-truncated and an isomorphism in degree $d$. 
\end{theorem}

This theorem allows us to deduce a calculation of the $K$-groups of the $p$-adic unit disc over a curve relative to a perfectoid ring with enough roots of unity, in terms of syntomic cohomology. This extends \cite[Corollary 4.17]{riggenbach2022cusps} to one relative dimension higher. 

\begin{theorem}\label{thm:main-curves}
    Let $C$ be the $p$-adic completion of a smooth affine curve over $R$, where $R$ is a perfectoid $\mathbb{Z}_p^{cycl}$-algebra. Then:
    \[
    K_{n}(C\langle t\rangle, (t);\Z_p) =\begin{cases}
    H^0(C\langle t\rangle;(\mathbb{G}_m)^\wedge_p)/H^0(C;(\mathbb{G}_m)^\wedge_p) & n = 2i-1, i \geq 1,\\
     H^2(C; N\mathbb{Z}_p(2)^{syn}) & n = 2i, i \geq 1.
    \end{cases} 
    \]
 
\end{theorem}

\subsection{Methods} There are, by now, more than one way to perform calculations of $K$-theory using trace methods. For example, one can redo Hesselholt-Madsen's classic calculation of $K$-theory of truncated polynomial rings over a perfect field of characteristic $p > 0$ \cite{hesselholt-madsen-trunc} via Nikolaus-Scholze's formulation of TC \cite{nikolaus2017topological} as performed by Speirs in \cite{speirs-trunc} or one can access it via prismatic cohomology as in Mathew \cite{mathew-recent} and Sulyma \cite{sulyma}.

In part, our paper proceeds similarly by favoring the Nikolaus-Scholze formula and using it to directly access a $TC$-theoretic and, subsequently, a $K$-theoretic statement. However, we do appeal to the motivic filtration and eventually prismatic cohomology to justify a technical lemma (see Lemma~\ref{lem: tate is coherent}) which passes our results from polynomial rings to smooth algebras. To carry this out, we improve Mathew's solution of the Segal conjecture for smooth over torsion-free perfectoid rings \cite{Mathew2021TR} in  \S\ref{sec:segal}. Our method, however, differs from his and follows more closely Hesselholt's method in \cite{lars-hasse-weil}. We then go from the Segal conjecture to a calculation of NTC by improving the second author's observation that one can commute certain direct sums against homotopy fixed and Tate constructions as long as we ``Nygaard complete'' the answers; see Lemma~\ref{lem: fixed points direct sum commute kinda}. Having these two facts, we deduce our main results in \S\ref{sec:ql} using fairly standard methods.  We also remark that Theorems~\ref{thm:main} and~\ref{thm:main-curves} also work for smooth, possibly non-affine, schemes over $R$ thanks to descent and Lemma~\ref{lem: tate is coherent}.

\subsection{Notation} We freely use the modern treatment of topological cyclic homology as in \cite{nikolaus2017topological}, its attendant motivic filtration as produced by Bhatt-Morrow-Scholze \cite{bhatt2019topological} and the theory of prisms and prismatic cohomology following \cite{prismatic, Bhatt-Scholze}. We also implicitly fix a prime $p$ throughout this paper. In particular, we use the following standard names for elements that frequently appear in this theory.
\begin{enumerate}
\item if $R$ is a perfectoid ring, we use $u \in TC_{2}(R;\ZZ_p)$ for the B\"okstedt generator and $\sigma \in TP_2(R;\ZZ_p)$ be the inverse of the chern class in $TP_{-2}(R;\ZZ_p)$.
\item if $(A, I)$ is an oriented prism, we use use $\xi$ for the (Frobenius inverse) of the generator of the Hodge-Tate divisor, i.e., we have that $(\phi(\xi)) = I$.
\end{enumerate}

\subsection{Acknowledgements} The first author is grateful to Lars Hesselholt for numerous discussions in and around topological cyclic homology of the affine line, starting from when he was a postdoc at Copenhagen. The authors are grateful to Akhil Mathew and Matthew Morrow for helpful discussions about the subject matter of this paper. The second author is supported by the Simons Collaboration on Perfection.

\section{\texorpdfstring{Some \emph{a priori} bounds for $p$-torsion free rings}{\emph{A priori} bounds}}

As stated in the introduction, our results on $NTC(A;\mathbb{Z}_p)$, where $A$ is a smooth-over-perfectoid, are better than one should expect. In this section we record a Quillen-Lichtenbaum style isomorphism between $TC$ and $L_{K(1)}TC$ for some rather general $A$'s. 

Towards this end let us recall some preliminaries. According to \cite{Bhatt_Mathew_fsmooth} a $p$-quasisyntomic ring $A$ is said to be \textbf{$F$-smooth} if two things happen: first that $\Prism_A\{ i\}$ is Nygaard complete and the fibre of the map $\mathcal{N}^i\Prism_A \rightarrow \overline{\Prism}_A\{i\}$ has $p$-complete tor-amplitude in degrees $\geq i+2$. The notion of $F$-smoothness is a non-noetherian extension of regularity as justified by \cite[Theorem 4.15]{Bhatt_Mathew_fsmooth}; it is also partly inspired by connectivity bounds expected out of the Segal conjecture for $THH$. 

The second author has proposed a definition for the dimension of an $F$-smooth ring which interacts well with TC-theoretic considerations. An $F$-smooth ring $A$ has \textbf{Nygaard dimension $\leq n$} (denoted by $\dim_{\mathcal{N}}A \leq n$) \cite[Definition 2.14]{riggenbach2023ktheorytruncatedpolynomials} if for all $i \in \Z$, we have that
\[
\mathcal{N}^{\geq i}\Prism_A \in D^{[0, n]}(\Z_p). 
\]
The Quillen-Lichtenbaum estimates that we will produce will depend on the Nygaard dimension of $A$. 

The last ingredient we need is a motivic filtration on $K(1)$-local $TC$ constructed by H. Kim \cite{kim}. This motivic filtration behaves as follows: for any (animated) ring $A$ the natural map $\TC(A;\Z_p) \rightarrow L_{K(1)}\TC(A)$ lifts to a multiplicative, filtered map
\[
F^{\star}_{\mot}\TC(A;\Z_p) \rightarrow F^{\star}_{\mot}L_{K(1)}\TC(A)
\]
such that, after taking graded pieces, we obtain a multiplicative graded map
\[
\Z_p(\star)^{\syn}(A^\wedge_p)[2\star] \rightarrow R\Gamma_{\et}(A^\wedge_p[\tfrac{1}{p}]; \Z_p(\star))[2\star],
\]
which agrees with the \'etale comparison map in prismatic theory by \cite[Corollary 3.36]{kim}. Therefore, we may appeal to \cite[Theorem 1.8]{Bhatt_Mathew_fsmooth} and note that for any $n \in \Z$ and any $k \geq 1$ the map on graded pieces
\[
\mathrm{gr}^{n}_{\mot}\TC(A;\Z_p)/p^k \rightarrow \mathrm{gr}^{n}_{\mot}L_{K(1)}\TC(A)/p^k
\]
is $n+1$-truncated. Our argument follows that of \cite[Theorem 6.13]{clausen2019hyperdescent} closely. 

\begin{theorem}
    Let $A$ be a $p$-torsion free F-smooth quasisyntomic ring with finite $\mathcal{N}$-dimension $\dim_{\mathcal{N}}(A)=d$. Then the map \[TC(A;\mathbb{F}_p)\to L_{K(1)}TC(A)/p\] is $d$-truncated.
\end{theorem}

\begin{proof}

Let $\mathcal{F}$ be the fibre of the functor $TC(-;\mathbb{F}_p) \rightarrow L_{K(1)}TC(-)/p$. Noting that this is a morphism of localizing invariants, we see that it admits pushforward along proper morphisms of finite tor-amplitude.
%

   This structure allows us to assume, without loss of generality, that $A$ has a $p^{th}$ root of unity since if $p=2$ then $-1\in A$ and if $p$ is odd then $\mathcal{F}(A)$ is a summand of $\mathcal{F}(A[\zeta_p])$ for any $p$-complete functor $\mathcal{F}$ with transfers (since we are working with mod-$p$ coefficients and the extension is coprime to $p$). Consequently we may ignore Tate twists: $R\Gamma_{\textrm{\'et}}(A[1/p];\mu_p^{\otimes n})\simeq R\Gamma_{\textrm{\'et}}(A[1/p];\mathbb{F}_p)$ for all $n\in \mathbb{Z}$. We then have that $\tau^{\leq n-1}R\Gamma_{\textrm{\'et}}(A[1/p];\mathbb{F}_p)\simeq \tau^{\leq n-1}\mathbb{F}_p(n)^{\mathrm{syn}}(A)$ which does not have any cohomology above degree $d+1$ by the assumption on Nygaard dimension. Since this is true for all $n$ it follows that $R\Gamma_{\textrm{\'et}}(A[1/p];\mathbb{F}_p)$ has cohomology in degrees $\leq d+1$.  

    It then follows that the map $\mathbb{F}_p(n)^{\mathrm{syn}}(A)\to R\Gamma_{\textrm{\'et}}(A[1/p];\mu_{p}^{\otimes n})$ is an equivalence when $n\geq d+2$. Inducing the motivic filtration on $\mathcal{F}$, we get that 
    \[
    \mathrm{gr}^n\mathcal{F}(A)\simeq \mathrm{fib}(\mathbb{F}_p(n)^{\mathrm{syn}}(A)\to R\Gamma_{\textrm{\'et}}(A[1/p];\mu_{p}^{\otimes n}))[2n].
    \] Consequently 
    \[
    \mathrm{gr}^n\mathcal{F}(A;\mathbb{Z}/p^k)=0 \qquad n\geq d+2,
    \] and the highest nonzero homotopy group of $\mathcal{F}(A)$ comes from 
    \begin{align*}\mathrm{gr}^{d+1}\mathcal{F}(A;\mathbb{F}_p) & \cong & H^{d+1}_{\textrm{\'et}}(A[1/p];\mu_{p}^{\otimes d+1})/H^{d+1}(\mathbb{F}_p(d+1)(A))[2(d+1)-(d+1)-1]\\ & \cong & H^{d+1}_{\textrm{\'et}}(A[1/p];\mu_{p}^{\otimes d+1})/H^{d+1}(A;\mathbb{F}_p(d+1))[d].
    \end{align*}
    The shift that appears in the last term tells us that the $\mathcal{F}$ has highest homotopy group in degree exactly $d$, whence the theorem is proved.
\end{proof}

\begin{remark}\label{rem:akhil} Notice that if we know generation by symbols, the last homotopy group also vanishes and we get that the map $TC(A;\mathbb{F}_p)\to L_{K(1)}TC(A;\mathbb{F}_p)$ is $(d-1)$-truncated. We would like to thank Akhil Mathew for pointing this out to us.
\end{remark}

\section{The Segal conjecture}\label{sec:segal}

In this Section we will prove the main results for the special case of $A=R\langle t_1,\ldots, t_d\rangle$. We will first prove in Subsection~\ref{ssec: coconnectivity of the frobenius} a version of the Segal conjecture for this class of rings. Some of this is covered in \cite{Mathew2021TR}, specifically the case when $R$ is $p$-torsion free \cite[Proposition 5.10]{Mathew2021TR}. Using this and an explicit formula for the topological negative cyclic and periodic homologies of $A$ we will be able to show the improved bound by hand in Subsection~\ref{ssec: the sharper bound for polynomial algebras}. 

\subsection{Coconnectivity of the Frobenius and the Segal conjecture}\label{ssec: coconnectivity of the frobenius}
In this section we will show that, in certain cases, the Segal conjecture survives taking a polynomial extension.  In other words, if the cyclotomic Frobenius is truncated for $\thh(A;\ZZ_p)$, then it will remain so for $\thh(A[t];\ZZ_p)$ with slightly worse bounds. We first recall that as a functor $\THH:\CAlg\to \cycsp$ is symmetric monoidal. In particular we have that for any $A \in \CAlg$,  

\begin{equation}\label{eq:kunneth}
\THH(A[t])\simeq \THH(A)\otimes \THH(\mathbb{S}[t])
\end{equation}
 and the Frobenius is given by $l\circ (\phi_p^{\THH(A)}\otimes \phi_p^{\THH(\mathbb{S}[t])})$ where $l$ is the map witnessing the lax monoidal structure of $(-)^{tC_p}$. The term $\THH(\mathbb{S}[t])$, and its attendant cyclotomic structure, has been known to experts for several years, and was recently carefully computed by \cite{mccandless2021curves} as part of their development of topological resriction homology in the framework set up by Nikolaus and Scholze in \cite{nikolaus2017topological}. 

As a $\mathbb{T}$-equivariant spectrum we have an equivalence 
\begin{equation}\label{eq:thh-st}
\THH(\mathbb{S}[t])\simeq \mathbb{S}\oplus \bigoplus_{j\geq 1} \Sigma^{\infty}_+\TT/C_j
\end{equation}
and the Frobenius is given by the sum of the composites:
\begin{equation}\label{eq:frob-cj}
\Sigma^{\infty}_+\TT/C_j\simeq \Sigma^{\infty}_+ \left(\TT/C_{pj}\right)^{C_p}\to\Sigma^{\infty}_+ \left(\TT/C_{pj}\right)^{hC_p}\to \left(\Sigma^{\infty}_+ \TT/C_{pj}\right)^{hC_p}\to\left(\Sigma^{\infty}_+ \TT/C_{pj}\right)^{tC_p} \to \THH(\mathbb{S}[t])^{tC_p}.
\end{equation}
Combining~\eqref{eq:kunneth} and~\eqref{eq:thh-st} we get an equivalence of $\mathbb{T}$-equivariant spectra
\[
\THH(A[t]) \simeq \THH(A) \oplus \bigoplus_{j\geq 1} \THH(A) \otimes \Sigma^{\infty}_+\TT/C_j.
\]
The cyclotomic Frobenius seems a little unwieldy as it is a tensor product of both the Frobenius on $\THH(A)$ and~\eqref{eq:frob-cj}. To facilitate our computations it will be helpful to prove that, after $p$-completion, the cyclotomic Frobenius on any individual summand only depends on the cyclotomic Frobenius on $\THH(A)$. This is made precise by the commutativity of the following diagram:
\begin{equation}\label{eq:termwise-frob}
\begin{tikzcd}
\THH(A) \otimes \Sigma^{\infty}_+\TT/C_j \ar[swap]{d}{\phi_p^{\THH(A)} \otimes\phi_p^{\THH(\mathbb{S}[t])}} \ar[bend left]{ddr}{\phi_p^{\THH(A)} \otimes \id} & \\
\THH(A)^{tC_p} \otimes (\Sigma^{\infty}_+\TT/C_{pj})^{tC_p} \ar[swap]{d}{l} & \\
(\THH(A) \otimes (\Sigma^{\infty}_+\TT/C_{pj})^{tC_p} & \THH(A)^{tC_p} \otimes \Sigma^{\infty}_+\TT/C_j \ar{l}{\id \otimes \widetilde{\phi}}.\\
\end{tikzcd}
\end{equation}
Its commutativity follows from \cite[Lemma 2]{hesselholt2019algebraic}; note that it still commutes after replacing $\THH(A)$ with its $p$-completion $\THH(A; \Z_p)$ in which case, the bottom map is an equivalence by the classical Segal conjecture. Consequently, up to equivalence, the Frobenius on $\THH(A[t];\ZZ_p)$ is given by the sum of the maps \[\thh(A;\ZZ_p)\otimes \Sigma^{\infty}_+\TT/C_j \xrightarrow{\phi_p^{\THH(A;\ZZ_p)}\otimes id}\thh(A;\ZZ_p)^{tC_p}\otimes \Sigma^{\infty}_+\TT/C_j\]

Next, we will need to understand how the Tate construction interacts with our direct sum decomposition; this is the technical heart of the paper. A priori there is no reason to expect this interaction to be nice, but it turns out to almost commute with the sum up to a completion.
\begin{lemma}\label{lem: fixed points direct sum commute kinda}
Fix a perfectoid ring $R$. Let $A$ be a formally smooth $R$-algebra or, more generally, any $R$-algebra such that $\hh(A/R)^\wedge_p$ is concentrated in finitely many degrees. Then for any set $I$, pointed $\TT$-spaces $B_i$ with no cells above dimension $k$ for some $k$(independent of $i$), and any $G=C_j$ or $\TT$ the maps \[\left(\widehat{\bigoplus}_{i\in I}\left(\thh(A;\ZZ_p)\otimes B_i\right)^{hG}\right)^\wedge_{p}\to \left(\left(\bigoplus_{i\in I}\thh(A;\ZZ_p)\otimes B_i\right)^{hG}\right)^\wedge_p\]
and 
\[\left(\bigoplus_{i\in I}\left(\thh(A;\ZZ_p)\otimes B_i\right)^{tG}\right)^\wedge_{(p,\xi)}\to \left(\left(\bigoplus_{i\in I}\thh(A;\ZZ_p)\otimes B_i\right)^{tG}\right)^\wedge_p\]
are equivalences. The notation $\widehat{\bigoplus}_{i\in I}$ denotes the Nygaard completion of the direct sum, or equivalently the $u$-adic completion for $u\in \TC^-_2(R;\ZZ_p)$.
\end{lemma}
\begin{proof}
From \cite[Theorem 6.7]{bhatt2019topological} we have cofiber sequences \[\Sigma^2\thh(A;\ZZ_p)\otimes B_i\xrightarrow{u}\thh(A;\ZZ_p)\otimes B_i\to \hh(A/R;\ZZ_p)\otimes B_i\] and by our hypotheses $\hh(A/R;\ZZ_p)\otimes B_i$ are uniformly bounded spectra. Inductively it then follows that the spectra \[\left(\thh(A;\ZZ_p)\otimes B_i\right)/u^n:=\mathrm{cof}\left(\Sigma^{2n}\thh(A;\ZZ_p)\otimes B_i\xrightarrow{u^n}\thh(A;\ZZ_p)\otimes B_i\right)\] are bounded uniformly in $i$. Since each $B_i$ are connective we also have that \[\lim_n\bigoplus_{i\in I}\left(\thh(A;\ZZ_p)\otimes B_i\right)/u^n\simeq \bigoplus_{i\in I}\thh(A;\ZZ_p)\otimes B_i\] since on any given homotopy groups the inverse limit stabilizes. This is a description as Borel $\TT$ spectra, so we get equivalences
\begin{align*}
    \left(\bigoplus_{i\in I}\thh(A;\ZZ_p)\otimes B_i\right)_{hG} &\simeq \left(\lim_n\bigoplus_{i\in I}\thh(A;\ZZ_p)/u^n\otimes B_i\right)_{hG}\\
                                                                &\simeq \lim_n\bigoplus_{i\in I}\left(\thh(A;\ZZ_p)/u^n \otimes B_i\right)_{hG}
\end{align*}
where the second equivalence comes from the fact that $\bigoplus_{i\in I}\thh(A;\ZZ_p)\otimes B_i\to \bigoplus_{i\in I}\thh(A;\ZZ_p)/u^n\otimes B_i$ is $(2n-1)$-connected and that homotopy orbits can only increase this. We also have equivalences 
\begin{align*}
        \left(\bigoplus_{i\in I}\thh(A;\ZZ_p)\otimes B_i\right)^{hG} &\simeq \left(\lim_n\bigoplus_{i\in I}\thh(A;\ZZ_p)/u^n\otimes B_i\right)^{hG}\\
                                                                &\simeq \lim_n\bigoplus_{i\in I}\left(\thh(A;\ZZ_p)/u^n \otimes B_i\right)^{hG}
\end{align*}
where the second equivalence comes from the uniform bound on $\thh(A;\ZZ_p)/u^n\otimes B_i$ and \cite[Lemma 3.1.3]{MR4326928}. The norm map will be the limit of the norm maps and so the same description holds for the Tate construction as well. 

To see the statement for the Tate construction it then remains to see that the $u$-adic and $\xi$-adic completions agree. To this end note that we have equivalences \[\left(\thh(A;\ZZ_p)/u^n \otimes B_i\right)^{tG}\simeq \left(\thh(A;\ZZ_p)\otimes B_i\right)^{tG}/u^n\] where the multiplication by $u^n$ comes from the lax-monoidal structure of $(-)^{tG}$ and the fact that $can: \TC^{-}(R;\ZZ_p)\to \mathrm{TP}(R;\ZZ_p)$ is an $\mathbb{E}_\infty$-ring map. In particular $u^n$ acts via $can(u^n)=\xi^n\sigma^n$ by \cite[Proposition 6.3]{bhatt2019topological}. In particular \[\left(\thh(A;\ZZ_p)\otimes B_i\right)^{tG}/u^n\simeq \left(\thh(A;\ZZ_p)\otimes B_i\right)^{tG}/\xi^n\] and the result follows.
\end{proof}

We are now able to prove the main result of this section. It states that the Segal conjecture survives taking a polynomial extension, at least at the cost of one extra homotopy group. 

\begin{proposition}\label{prop: coconnectivity of frobenius}
Fix a perfectoid ring $R$. Let $A$ be an $R$-algebra such that $\hh(A/R)^\wedge_p$ is concentrated in finitely many degrees. Suppose that $\phi_p:\thh(A;\ZZ_p)\to \thh(A;\ZZ_p)^{tC_p}$ is $k$-truncated and that the $k^{\textrm{th}}$ homotopy group of the fiber has bounded $p^\infty$-torsion. Then $\phi_p:\thh(A[t];\ZZ_p)\to \thh(A[t];\ZZ_p)^{tC_p}$ is $(k+1)$-truncated. The statement is still true with $\ZZ_p$ replaced with $\FF_p$.
\end{proposition}
\begin{proof}
From Lemma~\ref{lem: fixed points direct sum commute kinda} and the discussion preceeding it we have that \[\thh(A[t];\ZZ_p)^{tC_p}\simeq \thh(A;\ZZ_p)^{tC_p}\oplus \left(\bigoplus_{j\geq 1}\left(\thh(A;\ZZ_p)\otimes \Sigma^\infty_+\TT/C_j\right)^{tC_p}\right)^{\wedge}_{(p,\xi)}\] and on each summand of the source the Frobenius is given by \[\thh(A;\ZZ_p)\otimes \Sigma^\infty_+\TT/C_j\xrightarrow[]{\phi_p\otimes id}\thh(A;\ZZ_p)^{tC_p}\otimes \Sigma^\infty_+\TT/C_j\] followed by an equivalence $\thh(A;\ZZ_p)^{tC_p}\otimes \Sigma^\infty_+\TT/C_j\simeq \left(\thh(A;\ZZ_p)\otimes \Sigma^\infty_+\TT/C_{pj}\right)^{tC_p}$ and then the inclusion of this factor into $\thh(A[t];\ZZ_p)^{tC_p}$. 

We claim that $\left(\bigoplus_{j\geq 1}\left(\thh(A;\ZZ_p)\otimes \Sigma^\infty_+\TT/C_j\right)^{tC_p}\right)^\wedge_p$ is already $\xi$-adically complete. For this it is enough to show that each $\left(\bigoplus_{j\geq 1}\left(\thh(A;\ZZ_p)\otimes \Sigma^\infty_+\TT/C_j\right)^{tC_p}\right)/p^n$ is already $\xi$-adically complete, but then by induction it is enough to show that $\left(\bigoplus_{j\geq 1}\left(\thh(A;\ZZ_p)\otimes \Sigma^\infty_+\TT/C_j\right)^{tC_p}\right)/p$ is $\xi$-adically complete. For this we have that $\phi(\xi):=\xi^p+p\delta(\xi)$ and so being $\xi$-adially complete is equivalent to being $\phi(\xi)$-adically complete modulo $p$. In fact $\left(\bigoplus_{j\geq 1}\left(\thh(A;\ZZ_p)\otimes \Sigma^\infty_+\TT/C_j\right)^{tC_p}\right)/p$ is $\phi(\xi)$ torsion as a $\thh(R;\ZZ_p)^{tC_p}\simeq \mathrm{TP}(R;\ZZ_p)/\phi(\xi)$-module by \cite[Proposition 6.4]{bhatt2019topological}. 

Let $\mathcal{F}$ be the fiber of the Frobenius $\phi_p:\thh(A;\ZZ_p)\to \thh(A;\ZZ_p)^{tC_p}$. We then have that the fiber of the Frobenius $\phi_p:\thh(A[t];\ZZ_p)\to \thh(A[t];\ZZ_p)^{tC_p}$ is given by the $p$-adic completion of the spectrum \[\mathcal{F}\oplus \bigoplus_{j\geq 1}\left(\mathcal{F}\otimes \Sigma^{\infty}_+\TT/C_j\right)\] up to equivalence, because we know that the Frobenius is computed termwise by the commutativity of~\eqref{eq:termwise-frob}. By assumption this will not have homotopy above degree $k+1$, and in degree $k+1$ it will have bounded $p^\infty$ torsion. Consequently the $p$-completion will also not have homotopy above degree $k+1$ as desired. The statement mod $p$ is similar.
\end{proof}

\begin{remark}\label{rem:key} The key point of Proposition~\ref{prop: coconnectivity of frobenius}, as already observed in \cite{riggenbach2022cusps}, is that the $p$-completion of the big direct sum is $\xi$-adically complete. This lets us compute the fibre on polynomial extensions without additional $\xi$-completion, leading to stability  of the Segal conjecture under poiynomial extensions.
\end{remark}

One consequence of this is the following result.

\begin{corollary}
For each $d\geq 0$ and perfectoid ring $R$, the Frobenius \[\phi_p:\thh(R[t_1,\ldots, t_d];\ZZ_p)\to \thh(R[t_1,\ldots, t_d];\ZZ_p)^{tC_p}\] is $(d-3)$-truncated. If $R$ is $p$-torsion free then the Frobenius mod $p$ is still $(d-3)$-truncated.
\end{corollary}
\begin{proof}
The base case $d=0$ is \cite[Proposition 6.4]{bhatt2019topological}. The induction step is given by the above.
\end{proof}

\section{Quillen-Lichtenbaum for NTC}\label{sec:ql}

In this section, we prove Theorem~\ref{thm:main}. We outline the required steps:

\begin{enumerate}
\item we prove the result for polynomial rings in equicharacteristic $p$ in Lemma~\ref{lem: vanishing of NTC groups polynomial fp case.};
\item from which we can deduce the result for polynomial rings over torsion free perfectoid rings in Corollary~\ref{cor: NTC of ptf perfectoid polynomial algebra is k(1)-local}.
\item Excision then lets us boostrap the result for polynomial rings over an arbitrary perfectoid ring in Lemma~\ref{lem: main result for polynomial algebras};
\item Finally we prove the main result in Theorem~\ref{lem: tate is coherent}, relying on an \'etale base change result for the Tate construction in Lemma~\ref{lem: tate is coherent}.
\end{enumerate}

\subsection{The case of polynomial algebras}\label{ssec: the sharper bound for polynomial algebras}
First, we work with polynomial algebras over a perfect ring of characteristic $p > 0$. We note that for any $\F_p$-algebra $R$ we have that $L_{K(1)}TC(R) \simeq 0$. One way to see this is to observe that $L_{K(1)}TC(R)$ is a module over $K(\FF_p;\ZZ_p) \simeq \ZZ_p$ and therefore has vanishing $K(1)$-localization.

%

\begin{lemma}\label{lem: vanishing of NTC groups polynomial fp case.}
Let $R$ be a perfect $\FF_p$-algebra. Then $N\TC_i(R[t_1,\ldots, t_d])$ vanishes for $i \geq d$.
\end{lemma}
\begin{proof}
By Lemma~\ref{lem: fixed points direct sum commute kinda} we have that \[N\TC^{-}(R[t_1,\ldots, t_d])\simeq \Sigma\left(\bigoplus_{j\geq 1} \thh(R[t_1,\ldots, t_d])^{hC_j}\right)^{\wedge}_{(p,u)}\] and \[N\tp(R[t_1,\ldots, t_d])\simeq \Sigma\left(\bigoplus_{j\geq 1}\thh(R[t_1,\ldots, t_d])^{tC_j}\right)^\wedge_{(p,\xi)}\] where we are replacing $\left(X\otimes (\TT/C_j)_+\right)^{h\TT}$ with $\Sigma X^{hC_j}$ using the Wirthm\"uller  isomorphism. By \cite[Lemma 4.5]{riggenbach2022cusps} we may replace the $u$-adic completion on the topological negative cyclic homology with the $\xi$-adic filtration, and since $R$ is a perfect $\FF_p$-algebra this amounts to only taking the $p$-adic completion. Thus it is enough to show that the map \[\Sigma(\phi_p[p]-can): \Sigma \bigoplus_{j\geq 1}\thh(R[t_1,\ldots, t_d])^{hC_j}\to \Sigma \bigoplus_{j\geq 1}\thh(R[t_1,\ldots, t_d])^{tC_j}\] is an equivalence in degrees above $d-1$ and injective in degree $d-1$. Since $\phi_p$ multiplies the summand index by $p$ and $can$ preserves it we have for degree reasons that this map is injective, respectively surjective, whenever $\phi_p$ is by Lemma~\ref{lem: inj/surj/bij for degree reasons arg}. Since $\phi_p$ is $(d-3)$-truncated by Proposition~\ref{prop: coconnectivity of frobenius} the result follows.
\end{proof}

We will make use of the following statement in several places going forward, so we include here a proof for the convenience of the reader.
\begin{lemma}\label{lem: inj/surj/bij for degree reasons arg}
    Let $\{G_i\}_{i\in \mathbb{N}}$ and $\{H_i\}_{i\in \mathbb{N}}$ be two $\mathbb{N}$-indexed families of abelian groups, $f_i:G_i\to H_{i}$ and $g_i:G_i\to H_{i+1}$ be $\mathbb{N}$-indexed group homomorphisms. Suppose that $H_0=0$. Define $h:\bigoplus_{i\in \mathbb{N}}G_i\to \bigoplus_{i\in \mathbb{N}}H_i$ to be the map $g_i-f_i$ on the summand $G_i$. Then $h$ is injective/surjective/bijective if $g_i$ is for all $i\in \mathbb{N}$.
\end{lemma}
\begin{proof}
First suppose that $g_i$ is injective for all $i\in \mathbb{N}$, and let $(x_1,\ldots)\in \bigoplus_{i\in \mathbb{N}}G_i$ be an element of the kernel of $h$. Since we are working with the direct sum, unless $(x_1,\ldots)=0$ there is some $N$ large enough so that $x_N\neq 0$ but $x_K=0$ for all $K>N$. Note then that $h(x_1,\ldots)_{N+1}=g_N(x_N)-f_{N+1}(x_{N+1})=g(x_N)\neq 0$ since $g_N$ is injective. Hence $(x_1,\ldots)=0$ and $h$ is injective as desired. 

Now suppose that $g_i$ is surjective for all $i\in \mathbb{N}$ and let $(y_1,\ldots)\in \bigoplus_{i\in \mathbb{N}}H_i$. Since we are again working with the direct sum, either $(y_1,\ldots)=0$ which is clearly in the image or there is some $N$ large enough so that $y_N\neq 0$ and $y_K=0$ for all $K>N$. We will define a pre-image $(x_1,\ldots)$ of $(y_1,\ldots)$ inductively starting with $x_{K}=0$ for all $K\geq N$ and $x_{N-1}$ any pre-image of $y_N$ under $g_{N-1}$. Then define $x_{N-2}=\Tilde{y}_{N-1}+\widetilde{f_{N-1}(x_{N-1})}$ where the $\Tilde{(-)}$ denotes choosing a pre-image under $g_{N-2}$. Inductively we can always correct the error introduced by $f_i$ in this fashion.
\end{proof}

We now turn our attention to the case when $R$ is $p$-torsion free. In this case we still have the computations \[N\TC^{-}(R[t_1,\ldots, t_d];\FF_p)\simeq \Sigma\left(\bigoplus_{j\geq 1} \thh(R[t_1,\ldots, t_d];\FF_p)^{hC_j}\right)^{\wedge}_{\xi}\] and \[N\tp(R[t_1,\ldots, t_d];\FF_p)\simeq \Sigma\left(\bigoplus_{j\geq 1}\thh(R[t_1,\ldots, t_d];\FF_p)^{tC_j}\right)^\wedge_{\xi}\] by combining Lemma~\ref{lem: fixed points direct sum commute kinda} and \cite[Lemma 4.5]{riggenbach2022cusps}. A similar argument as in Lemma~\ref{lem: vanishing of NTC groups polynomial fp case.} together with the fact that $R$ being $p$-torsion free inductively shows that $\thh(R[t_1,\ldots, t_d];\ZZ_p)$ is $p$-torsion free. In turn, this implies that the induced map \[\Sigma\bigoplus_{j\geq 1} \thh(R[t_1,\ldots, t_d];\FF_p)^{hC_j}\to \Sigma\bigoplus_{j\geq 1} \thh(R[t_1,\ldots, t_d];\FF_p)^{tC_j}\] is $(d-2)$-truncated. If $\phi_p-can$ was $\xi$-linear this would then imply that we had the same style of vanishing as in Lemma~\ref{lem: vanishing of NTC groups polynomial fp case.} but we know from \cite[Theorem 1.2]{riggenbach2022cusps} that this does not happen.

In order to account for the effect that the $\xi$-adic completion has on the final answer we first restrict our attention to $R$ which have a compatible system of $p^{th}$ power roots of unity. In other words we are considering $R$ which are $\ZZ_p^{cycl}:= \ZZ_p[\zeta_{p^\infty}]^{\wedge}_p$-algebras. Recall the notation $\epsilon := (1,\zeta_p, \zeta_{p^2},\ldots)\in R^\flat$. From \cite[Example 3.16]{Bhatt2018Integral} we have that $ 1+[\epsilon^{{1}/{p}}]+[\epsilon^{{1}/{p}}]^2+\ldots +[\epsilon^{{1}/{p}}]^{p-1}$ is an orientation of $R$ where $[-]:R^\flat \to A_{inf}(R)$ is the Teichm\"uller lift; in other words this is the (Frobenius inverse) of the orientation for the prism corresponding to the perfectoid ring $R$ and we call it $\xi$. If we let $\mu:= [\epsilon]-1$ we then also have that $\Tilde{\xi}=\phi(\mu)/\mu$ and that $\xi=\mu/\phi^{-1}(\mu)$ by \cite[Proposition 3.17]{Bhatt2018Integral}. Then $\phi^{-1}(\mu)u\in \tc^{-}_{2}(R;\ZZ_p)$ is such that $can(\phi^{-1}(\mu)u)=\xi\phi^{-1}(\mu)\sigma = \mu\sigma=\phi_p(\phi^{-1}(\mu)u)$ in $\tp_{2}(R;\ZZ_p)$ by \cite[Proposition 6.2 and 6.3]{bhatt2019topological}. In particular by \cite{nikolaus2017topological} there exists a unique element $\beta\in \tc_2(R;\ZZ_p)$ lifting $\phi^{-1}(\mu)u$.

\begin{lemma}\label{lem: induction step for polynomial ptf}
Let $R$ be a $p$-torsion free perfectoid $\ZZ_p^{cycl}$-algebra. Then the map \[N\TC(R[t_1,\ldots, t_d];\FF_p)\to N\TC(R[t_1,\ldots, t_d];\FF_p)[1/\beta]\] is $(d-2)$-truncated.
\end{lemma}
\begin{proof}
Base changing the fiber sequence giving $N\tc(R[t_1,\ldots, t_d];\FF_p)$ shows that the fiber of the map in question is the total fiber of the square \[\begin{tikzcd}
{N\tc^{-}(R[t_1,\ldots, t_d];\FF_p)} \arrow[d] \arrow[r] & {N\tp(R[t_1,\ldots, t_d];\FF_p)} \arrow[d] \\
{N\tc^{-}(R[t_1,\ldots, t_d];\FF_p)[1/\beta]} \arrow[r]  & {N\tp(R[t_1,\ldots, t_d];\FF_p)[1/\beta]} 
\end{tikzcd}\] and by a similar argument as in \cite[Lemma 4.13]{riggenbach2022cusps} we may replace inverting $\beta$ with inverting $\xi$.

In order to compute the total fiber of the above square we will first compute the vertical fibers. To do this first note that $\thh(R[t_1,\ldots, t_d];\FF_p)^{hC_j}$ is a module over $\thh(R;\FF_p)^{hC_j}$ and similarly $\thh(R[t_1,\ldots, t_d];\FF_p)^{tC_j}$ is a module $\thh(R;\FF_p)^{tC_j}$. In particular since both $\thh(R;\FF_p)^{hC_j}$ and $\thh(R;\FF_p)^{tC_j}$ are $\xi$-power torsion so are $\thh(R[t_1,\ldots, t_d];\FF_p)^{hC_j}$ and $\thh(R[t_1,\ldots, t_d];\FF_p)^{tC_j}$. Consequently \[\left(\Sigma\bigoplus_{j\geq 1}\thh(R[t_1,\ldots, t_d];\FF_p)^{hC_j}\right)[1/\xi]\simeq 0\simeq \left(\Sigma\bigoplus_{j\geq 1}\thh(R[t_1,\ldots, t_d];\FF_p)^{tC_j}\right)[1/\xi]\] since inverting elements will commute with direct sums and shifts. 

Thus by an arithmetic fracture square argument the vertical fibers of the square in question are \[\Sigma\bigoplus_{j\geq 1}\thh(R[t_1,\ldots, t_d];\FF_p)^{hC_j}\to \Sigma\bigoplus_{j\geq 1}\thh(R[t_1,\ldots, t_d];\FF_p)^{tC_j}\] which we have already shown to be $(d-2)$-truncated in the paragraph preceding this Lemma.
\end{proof}

We are now able to prove the main result for polynomial rings over a $p$-torsion free perfectoid ring.
\begin{corollary}~\label{cor: NTC of ptf perfectoid polynomial algebra is k(1)-local}
Let $R$ be a $p$-torsion free perfectoid ring. Then the map \[\tc(R[t_1,\ldots, t_d];\FF_p)\to L_{K(1)}\tc(R[t_1,\ldots, t_d];\FF_p)\] is $\max\{-1, (d-3)\}$-truncated. 
\end{corollary}
\begin{proof}
Note that this statement is fpqc local in $R$ since $\tc(-[t_1,\ldots, t_d];\FF_p)$ satisfies fpqc descent and $L_{K(1)}(-)$ kills bounded above spectra. In particular we may assume without loss of generality that $R$ has a consistent system of $p^{th}$ power roots of unity. See the first paragraph of \cite[Proposition 5.10]{Mathew2021TR} for the detail of this argument.

We will prove this statement by induction on $d$. For $d=1$ note that by \cite[Proposition 5.10]{Mathew2021TR} we have that $\tc(R;\FF_p)\to L_{K(1)}\tc(R;\FF_p)$ is $(-1)$-truncated. Thus for $\tc(R[t];\FF_p)\to L_{K(1)}\tc(R[t];\FF_p)$ to be $d-2=-1$-truncated it is enough to show that $N\tc(R;\FF_p)\to L_{K(1)}N\tc(R)$ is $(-1)$-truncated. This is a consequence of \cite[Proposition 4.14]{riggenbach2022cusps} where we identify $N\tc(R;\FF_p)[1/\beta]\simeq L_{K(1)}N\tc(R;\FF_p)$ by \cite[Theorem 1.3.6, Lemma 1.3.7, and Corollary 1.3.8]{hesselholt2020tcreview}.

Now suppose that for some $d\in \NN$, \[\tc(R[t_1,\ldots, t_d];\FF_p)\to L_{K(1)}\tc(R[t_1,\ldots, t_d];\FF_p)\] $\max\{-1, (d-3)\}$-truncated. Then we have that \[\tc(R[t_1,\ldots, t_{d+1}];\FF_p)\simeq \tc(R[t_1,\ldots, t_d];\FF_p)\oplus N\tc(R[t_1,\ldots, t_d];\FF_p)\] and so in order to show that the map \[\tc(R[t_1,\ldots, t_{d+1}];\FF_p)\to L_{K(1)}\tc(R[t_1,\ldots, t_{d+1}];\FF_p)\] is $\max\{-1, d-2\}$-truncated it is enough to show this on each of the summands. The first summand is the induction hypothesis and the second summand is handled in Lemma~\ref{lem: induction step for polynomial ptf}. 
\end{proof}

We are now ready to prove the polynomial case for a general perfectoid ring $R$.

\begin{lemma}\label{lem: main result for polynomial algebras}
Let $R$ be a perfectoid ring. Then the map \[\tc(R[t_1,\ldots, t_d];\ZZ_p)\to L_{K(1)}\tc(R[t_1,\ldots, t_d];\ZZ_p)\] is $\max\{0, d-1\}$-truncated and an isomorphism in degree $\max\{1,d-1\}$. 
\end{lemma}
\begin{proof}
In \cite[Lemma 7.19]{Mathew2021TR} it is shown that every perfectoid ring $R$ fits into a square \[\begin{tikzcd}
R \arrow[r]\arrow[d] & R_0\arrow[d]\\
(R/p)_{perf} \arrow[r] & (R_0/p)_{perf}
\end{tikzcd}\]
which is a Milnor square and a pushout square of $\mathbb{E}_\infty$ rings; in particular it is motivic pullback square in the sense of \cite{land-tamme,land-tamme2}. Here $R_0=R/R[p]$ is a $p$-torsion free perfectoid ring and the rings on the bottom are perfect $\FF_p$-algebras. Since $R\to R[t_1,\ldots, t_d]$ is flat we have that \[\begin{tikzcd}
R[t_1,\ldots, t_d] \arrow[d]\arrow[r] & R_0[t_1,\ldots, t_d]\arrow[d]\\
(R/p)_{perf}[t_1,\ldots, t_d]\arrow[r] & (R_0/p)_{perf}[t_1,\ldots, t_d]
\end{tikzcd}\]
is again a pushout of $\mathbb{E}_\infty$ rings. It is also a Milnor square since adjoining polynomial generators will commute with taking the kernel of the horizontal maps. 

Applying \cite[Theorem A]{land-tamme} to the above square then gives a pushout square \[\begin{tikzcd}
\TC(R[t_1,\ldots, t_d];\ZZ_p) \arrow[d]\arrow[r] & \tc(R_0[t_1,\ldots, t_d];\ZZ_p)\arrow[d]\\
\tc((R/p)_{perf}[t_1,\ldots, t_d];\ZZ_p)\arrow[r] & \tc((R_0/p)_{perf}[t_1,\ldots, t_d];\ZZ_p)
\end{tikzcd}\] and a similar square for the $K(1)$ local topological cyclic homology. In particular from this or from \cite[Corollary 1.3]{bhatt2020remarks} we see that $L_{K(1)}\tc(R[t_1,\ldots, t_d])\simeq L_{K(1)}\tc(R_0[t_1,\ldots, t_d])$ (since the bottom terms vanish after $K(1)$-localization) and so it is enough, because of the torsion free case in Corollary~\ref{cor: NTC of ptf perfectoid polynomial algebra is k(1)-local}, to show that the map $\tc(R[t_1,\ldots, t_d];\ZZ_p)\to \tc(R_0[t_1,\ldots, t_d];\ZZ_p)$ is an equivalence after taking $\tau_{\geq \max\{1, d-1\}}$.

To this end note that there is a long exact sequence 
\[\begin{tikzcd}\ldots\arrow[r] & \tc_*(R[t_1,\ldots, t_d];\ZZ_p)\arrow[r] &\begin{matrix}\tc_*(R_0[t_1,\ldots, t_d];\ZZ_p)\\ \oplus\\ \tc_*((R/p)_{perf}[t_1,\ldots, t_d];\ZZ_p)\end{matrix}\arrow[ldd, out=-60]\\ \\ &\tc_*((R_0/p)_{perf}[t_1,\ldots, t_d];\ZZ_p)\arrow[r]& \tc_{*-1}(R[t_1,\ldots, t_d];\ZZ_p) \arrow[r]& \ldots\end{tikzcd}\] and we have that both $\tc((R/p)_{perf}[t_1,\ldots, t_d];\ZZ_p)$ and $\tc((R_0/p)_{perf}[t_1,\ldots, t_d];\ZZ_p)$ vanish in degrees $\geq \max\{1,d-1\}$ by Lemma~\ref{lem: vanishing of NTC groups polynomial fp case.}. 

\end{proof}

\subsection{Bootstrap to smooth schemes}
In this section we will prove the main results. This will follow the proofs of the polynomial cases. We will begin by showing the result in positive characteristic.

\begin{theorem}~\label{thm: main theorem in positive characteristic}
    Let $X$ be a quasi-compact smooth $k$-scheme of relative dimension $d$, $k$ a perfect $\mathbb{F}_p$-algebra. Then $NTC_i(X)=0$ for $i\geq d$.
\end{theorem}
\begin{proof}
    Since $X$ is smooth over $k$, each point  $x\in X$ admits an open neighborhood $U\subseteq X$ and an {\'e}tale morphism $U\to \mathrm{Spec}(k[x_1,\ldots, x_d])$ over $\mathrm{Spec}(k)$. It is then enough to show that each of these schemes $U$ has the desired vanishing property since $NTC$ is a Zariski sheaf. Note also that we may assume without loss of generality that the $U$ are affine.
    
    For a scheme $X$, denote by $\mathcal{F}(X)$ the fiber of the cyclotomic Frobenius $\phi_p:\thh(X)\to \thh(X)^{tC_p}$. By the same argument as in Lemma~\ref{lem: induction step for polynomial ptf} and Proposition~\ref{prop: coconnectivity of frobenius}, it is enough to show that $\tau_{\geq d-2}\mathcal{F}(U)=0$ and we have already shown that $\tau_{\geq d-2}\mathcal{F}(k[x_1,\ldots, x_d])=0$. The result will follow  from the fact that $\mathcal{F}(-)$ is a coherent {\'e}tale sheaf, in the sense that \[\mathcal{F}(U)\simeq U\otimes_{k[x_1,\ldots, x_d]}\mathcal{F}(k[x_1,\ldots, x_d]).\] Coherent {\'e}tale sheaves are closed under cofibers, so it is enough to show that this is true for $\thh(-)$ and $\thh(-)^{tC_p}$. That $\thh(-)$ is a coherent sheaf is by \cite[Lemma 2.4.2]{Hesselholt_Curves}. Let $f:A\to B$ be an {\'e}tale extension of $\mathbb{F}_p$-algebras. It is then enough to show that the map $B\otimes_A \thh(A)^{tC_p}\to \thh(B)^{tC_p}$ is an equivalence, where the $A$ and $B$ module structures come from the ring map $\thh(-)\xrightarrow{\phi_p}\thh(-)^{tC_p}$. Since we will need this result again in a slightly more general context we will prove this separately in Lemma~\ref{lem: tate is coherent}.
\end{proof}

The proof of Theorem~\ref{thm: main theorem in positive characteristic} makes use of the following observation, the importance of which the authors learned from Lars Hesselholt; see \cite[Propositions 6.2-6.4]{lars-hasse-weil} for this argument in a different context. 

\begin{lemma}~\label{lem: tate is coherent}
    Let $f:A\to B$ be an {\'e}tale map of smooth $R$-algebras, $R$ a perfectoid ring. Then the natural map \[B\otimes_A\thh(A)^{tC_p}\to \thh(B)^{tC_p}\] is an equivalence.
\end{lemma}
\begin{proof}
  In this context, let us recall from \cite{bhatt2019topological} that there is a motivic filtration $\mathrm{Fil}^{\star}_{\mot}\THH(A)^{tC_p}$ whose graded pieces are given by $\overline{\Prism}_{A/R}[2i]$ (ignoring Breul-Kisin twists in the perfectoid setting). Since $B\otimes_A (-)$ preserves connectivity and $\overline{\Prism}_{A/R}$ has no cohomology above the relative dimension of $A$ over $R$, we have that the motivic filtrations on both the source and target are complete and it is enough to check this isomorphism on associated graded terms. The Breuil-Kisin twists are trivial since we are working over a perfectoid ring and therefore we reduce to showing that $B\otimes_A \overline{\Prism}_{A/R}\to \overline{\Prism}_{B/R}$ is an equivalence which can be proved by reducing to a similar base change formula for the cotangent complex \cite[Tag 08R2-Tag 08R3]{stacks-project} or an appeal to \cite[Lemma 4.21]{Bhatt-Scholze} which uses the site-theoretic formalism. Note that the $A$-modules structure on $\overline{\Prism}_{A/R}$ is the correct one to apply this result by \cite[Theorem 1.14(2)]{Bhatt-Scholze}.
\end{proof}

We now turn to the $p$-torsion free case.

\begin{theorem}~\label{thm: main theorem ptf case}
    Let $X$ be a quasi-compact smooth scheme over $\mathrm{Spec}(R)$, $R$ a $p$-torsion free perfectoid ring. Then the map \[NTC(X;\mathbb{F}_p)\to L_{K(1)}NTC(X;\mathbb{F}_p)\] is $\max\{-1,(d-2)\}$-truncated.
\end{theorem}
\begin{proof}
    We will proceed in a very similar fashion as in the positive characteristic case. Note that since $L_{K(1)}-$ kills bounded above spectra, we get that $L_{K(1)}NTC(-;\mathbb{F}_p)$ is still a Zariski sheaf. Thus again by quasi-compactness and smoothness we may reduce to the case of $U\to \mathrm{Spec}(R[ x_1,\ldots, x_d ])$ an {\'e}tale cover with $U\cong \mathrm{Spec}(A)$. By fpqc descent we may also assume without loss of generality that $R$ contains a compatible system of $p^{th}$-power roots of unity. 

    Since we have a compatible system of $p^{th}$ power roots of unity in $R$, the same argument as in Corollary~\ref{cor: NTC of ptf perfectoid polynomial algebra is k(1)-local} shows that the fiber of the map $NTC(U;\mathbb{F}_p)\to L_{K(1)}NTC(U;\mathbb{F}_p)$ is given by the fiber of the map \[\Sigma \left(\bigoplus_{i\geq 1}\thh(A;\mathbb{F}_p)^{hC_i}\right)\to \Sigma\left(\bigoplus_{i\geq 1}\thh(A;\mathbb{F}_p)^{tC_i}\right)\] with the map given by $\Sigma(\phi_p^{hC_i}[p]-can^{hC_i})$. By Lemma~\ref{lem: inj/surj/bij for degree reasons arg}, this map will be (injection / surjection / isomorphism) on $\pi_*$ provided that $\phi_p^{hC_i}:\thh(A;\mathbb{F}_p)^{hC_i}\to \thh(A;\mathbb{F}_p)^{tC_{pi}}$ is an (injection / surjection / isomorphism) on $\pi_{*-1}$ for all $i$. Since homotopy orbits preserves coconnectivity of maps it is then enough to show that $\phi_p:\thh(A;\mathbb{F}_p)\to \thh(A;\mathbb{F}_p)^{tC_p}$ has fiber concentrated in degrees $\leq d-3$. This either follows by a similar analysis as in Theorem~\ref{thm: main theorem in positive characteristic} or by using \cite[Proposition 5.10]{Mathew2021TR}.
\end{proof}

We are now ready to prove Theorem~\ref{thm:main}, which we recall the statement here for convinience.

\begin{theorem}\label{thm:main-body}
    Let $X$ be a quasi-compact quasi-seperated smooth $R$-scheme of relative dimension $d$, $R$ a perfectoid ring. Then the map \[NTC(X;\mathbb{Z}_p)\to L_{K(1)}NTC(X)\] is an $(d-1)$-truncated and an isomorphism in degree $d$.
\end{theorem}
\begin{proof}
    Similar to the proof of Theorem~\ref{thm: main theorem ptf case}, we may reduce to the case of $X=\mathrm{Spec}(A)$ where $A$ is a smooth $R$-algebra admitting a factorization $R\to R[x_1,\ldots, x_d]\xrightarrow{f} A$ with $f$ {\'e}tale. 
    
    Let $R_0$ denote the perfecotid ring $R/(R[p])$, so that as in Lemma~\ref{lem: main result for polynomial algebras} we have a Milnor square 
    \[
    \begin{tikzcd}
        R \ar[d] \ar[r] & R_0\ar[d]\\
        (R/p)_{perf} \ar[r] & (R_0/p)_{perf}
    \end{tikzcd}
    \] of perfectoid rings where $R_0$ is $p$-torsion free and both $(R/p)_{perf}$ and $(R_0/p)_{perf}$ are perfect $\mathbb{F}_p$-algebras. Note that tensoring the above Milnor square with any flat $R$-algebra will again give a Milnor square since by flatnees tensoring with such a ring will preserve pullback squares \cite[Lemma 3.2.8]{elmanto2021cdh} and tensoring with any ring preserves pushout squares. Consequently if we define $A_0:= R_0\otimes_R A$, $A':= (R/p)_{perf}\otimes_R A$, and $A'_0:= (R_0/p)_{perf}\otimes_R A$ we that the square 
    \[
    \begin{tikzcd}
        A \ar[r] \ar[d] & A_0 \ar[d]\\
        A' \ar[r] & A'_0
    \end{tikzcd}
    \] is a Milnor square.

    By \cite[Theorem A]{land-tamme} we then have a pullback square
    \[
    \begin{tikzcd}
        NTC(A;\mathbb{Z}_p) \ar[r] \ar[d] & NTC(A_0;\mathbb{Z}_p) \ar[d]\\
        NTC(A';\mathbb{Z}_p) \ar[r] & NTC(A'_0;\mathbb{Z}_p)
    \end{tikzcd}
    \]
    and an equivalence $L_{K(1)}NTC(A;\mathbb{Z}_p)\simeq L_{K(1)}NTC(A_0;\mathbb{Z}_p)$. Taking $\mathcal{F}_{NTC}(-)$ to be the fiber of the localization map $NTC(-;\mathbb{Z}_p)\to L_{K(1)}NTC(-)$ we then have that the square
    \[
    \begin{tikzcd}
        \mathcal{F}_{NTC}(A)\ar[r] \ar[d] & \mathcal{F}_{NTC}(A_0)\ar[d]\\
        NTC(A';\mathbb{Z}_p) \ar[r] & NTC(A'_0;\mathbb{Z}_p)
    \end{tikzcd}
    \] is pullback. 

    Now, base change along a map preserves smoothness and can only lower relative dimension. Thus $A_0$ is a smooth $R_0$-algebra of relative dimension at most $d$, so by Theorem~\ref{thm: main theorem ptf case} we have that $\mathcal{F}_{NTC}(A_0)$ is $(d-2)$-truncated. In addition each of $NTC(A')$ and $NTC(A'_0)$ are $d-1$ truncated by Theorem~\ref{thm: main theorem in positive characteristic}, so we get that $\mathcal{F}_{NTC}(A)$ is $(d-1)$-truncated. 

    All that remains to show is that the map $\pi_d(NTC(A;\ZZ_p))\to \pi_d(L_{K(1)}NTC(A;\ZZ_p))$ is an isomorphism. Since the map $NTC(A;\ZZ_p)\to L_{K(1)}NTC(A)$ is $(d-1)$-truncated, we already have that the map on $\pi_d$ is injective. Thus to prove the result we only need to verify that $\pi_d(NTC(A;\mathbb{Z}_p))\to \pi_d(L_{K(1)}NTC(A))$ is surjective. This is accomplished in two steps. First, the map $\pi_d(NTC(A;\mathbb{Z}_p))\to \pi_d(NTC(A_0;\mathbb{Z}_p))$ is an equivalence. To see this, note that the pullback square above induces a long exact sequence \[\ldots \to \pi_{d+1}(NTC(A_0';\ZZ_p))\to \pi_d(NTC(A;\ZZ_p))\to \substack{\pi_d(NTC(A_0;\ZZ_p))\\ \oplus \\ \pi_d(NTC(A';\ZZ_p))}\to \pi_d(NTC(A_0';\ZZ_p))\to\ldots\] and $\pi_{d+1}(NTC(A_0';\ZZ_p))\cong \pi_d(NTC(A';\ZZ_p))\cong \pi_d(NTC(A_0';\ZZ_p))=0$  by Theorem~\ref{thm: main theorem in positive characteristic}. Thus We get the stated isomorphism. We also have that $L_{K(1)}NTC(A)\to L_{K(1)}NTC(A_0)$ is an equivalence since the fiber is truncated. Thus surjectivity follows from the fact that the diagram 
    \[
    \begin{tikzcd}
        \pi_d(NTC(A;\mathbb{Z}_p)) \ar[r, "\cong"] \ar[d] & \pi_d(NTC(A_0;\mathbb{Z}_p))\ar[d] \\
        \pi_d(L_{K(1)}NTC(A)) \ar[r, "\cong"] & \pi_d(L_{K(1)}NTC(A_0))
    \end{tikzcd}
    \]
    commutes and that the right hand vertical map is an equivalence since $\mathcal{F}_{NTC}(A_0)$ is $(d-2)$-truncated.
\end{proof}

\begin{remark}\label{rem:bounds} It is not straightforward to conclude the bounds on homotopy groups obtained in Theorem~\ref{thm:main-body} from Lemma~\ref{lem: main result for polynomial algebras} because the truncation functors do not preserve fibre sequences. Hence we need additional arguments as in the last paragraph of the previous proof. 
\end{remark}

\section{Application: $p$-adic unit disks of curves over perfectoid bases}

We now turn our attention to the $K$-theory applications of our result. Thus far we have only consider topological cyclic homology, and so our first step is to translate these results over to algebraic $K$-theory.
\begin{lemma}~\label{lem: TC to K theory translation}
    Let $A$ be the $p$-completion of a smooth $R$-algebra, $R$ a perfectoid ring. Suppose that $A$ has relative dimension $d$ over $R$. Then the map \[K(A\langle t\rangle, (t); \mathbb{Z}_p)\to TC(A\langle t\rangle, (t); \mathbb{Z}_p)\] is $\max\{d-1,0\}$-truncated and an equivalence in degree $d$. 
    \end{lemma}
\begin{proof}
    Contemplate the following commutative cube
    \[
    \begin{tikzcd}
        \tau_{\geq 0}K(A;\ZZ_p)\ar[rr] \ar[dd] \ar[rd]& & TC(A;\ZZ_p) \ar[dd] \ar[rd] & \\
         & \tau_{\geq 0}K(A\langle t\rangle; \ZZ_p) \ar[rr] \ar[dd] & & TC(A\langle t\rangle; \ZZ_p)\ar[dd]\\
         K(A/\sqrt{(p)};\ZZ_p) \ar[rr] \ar[dr, "\simeq"] & & TC(A/\sqrt{(p)};\ZZ_p)\ar[rd] & \\
        & K(A/\sqrt{(p)}[t];\ZZ_p) \ar[rr] & & TC(A/\sqrt{(p)}[t];\ZZ_p) & 
    \end{tikzcd}
    \]
    where $\sqrt{(p)}$ denotes the radical of the ideal $(p)$. The front and back faces are pullback square by \cite[Theorem A]{clausen2018k}. Thus the square of cofibers is also a pullback. We also have that the bottom left diagonal map is an equivalence since $A/\sqrt{(p)}$ is smooth over the perfect $\mathbb{F}_p$-algebra $R/\sqrt{(p)}$, and so we get a fiber sequence $\tau_{\geq 0}K(A\langle t\rangle, (t);\ZZ_p)\to TC(A\langle t\rangle, (t);\ZZ_p)\to NTC(A/\sqrt{(p)};\ZZ_p)$, and this cofiber term is $(d-1)$-coconnective by Theorem~\ref{thm:main-body}. The result then follow from the result on topological cyclic homology.
\end{proof}
Note that this Lemma also proves the last statement of Theorem~\ref{thm:main}.

In this Section we will demonstrate the utility of our results by studying the case of curves over perfectoid rings. We will first consider the characteristic $p$ case, where we can prove a vanishing result not just for the algebraic $K$-groups but also on the level of syntomic cohomology.

\begin{lemma}~\label{lem: vanishing in positive characteristic for curves}
    Let $C$ be a smooth curve over $k$, where $k$ is a perfect $\mathbb{F}_p$ algebra. Then $N\mathbb{Z}_p(2)^{syn}(C)=0$.
\end{lemma}
\begin{proof}
    Since $N\mathbb{Z}_p(2)^{syn}$ is a Zariski sheaf we may assume without loss of generality that $C$ is affine. It then follows that $N\mathbb{Z}_p(2)^{syn}(C)$ is concentrated in degrees $[0,3]$. Since $C$ is smooth it also follows that $N\mathbb{Z}_p(i)(C)=0$ for all $i\geq 3$. The only possible differential in the spectral spectral sequence \cite[Theorem 1.12(5)]{bhatt2019topological} is the differential $H^0(N\mathbb{Z}_p(1)^{syn}(C))\to H^3(N\mathbb{Z}_p(2)^{syn}(C))$. Again since $C$ is smooth $H^0(N\mathbb{Z}_p(1)^{syn}(C))=0$, so there are no differentials and each group $NTC_*(C)$ is given by an extension of the groups $H^i(N\mathbb{Z}_p(j)(C))$ where $2j-i=*$, $0\leq j\leq 2$, and $0\leq i\leq 3$. Thus the cohomology of $N\mathbb{Z}_p(2)^{syn}(C)$ contributes to $\tau_{\geq 1}NTC(C)$, but this vanishes by Theorem~\ref{thm: main theorem in positive characteristic}.
\end{proof}

Consider now $C$ the $p$-adic completion of a smooth curve over $R$, where $R$ is a general perfectoid ring. We then have that $\tau_{\geq 1}K(C\langle t\rangle, (t);\mathbb{Z}_p)\simeq \tau_{\geq 1}NTC(C;\mathbb{Z}_p)\simeq \tau_{\geq 1}L_{K(1)}NTC(C;\mathbb{Z}_p)$ by Theorem~\ref{thm:main}. In the case of $R$ a $\mathbb{Z}_p^{cycl}$ algebra we may then completely describe the $NK$ groups of $C$ in terms of more classical invariants.

\begin{theorem}
    Let $C$ be the $p$-adic completion of a smooth affine curve over $R$, where $R$ is a perfectoid $\mathbb{Z}_p^{cycl}$-algebra. Then there are isomorphisms \[K_{2i-1}(C\langle t\rangle, (t);\mathbb{Z}_p)\cong H^0(C\langle t\rangle;(\mathbb{G}_m)^\wedge_p)/H^0(C;(\mathbb{G}_m)^\wedge_p)\] and \[K_{2i}(C\langle t\rangle, (t);\mathbb{Z}_p)\cong H^2(C; N\mathbb{Z}_p(2)^{syn})\] for all $i\geq 1$.
\end{theorem}
\begin{proof}
    By the discussion preceding the theorem statement we have that $\tau_{\geq 1}K(C\langle t \rangle, (t);\mathbb{Z}_p)$ is the $1$-connective cover of a $L_{K(1)}K(\mathbb{Z}_p^{cycl})$-algebra, and therefore is $2$-periodic. Thus to prove the result we need to show that the computation works for $i=1$. Since $C$ is the $p$-adic completion of a smooth affine curve, $\dim_{\mathcal{N}}(C\langle t\rangle)\leq 2$ by the Hodge-Tate comparison theorem. Therefore we have that the complexes $N\mathbb{Z}_p(i)(C)$ have cohomology in degree at most $3$. Thus the only groups that can contribute to $NTC_1(C;\mathbb{Z}_p)$ are $H^1(C;N\mathbb{Z}_p(1)^{syn})$ and $H^3(C;N\mathbb{Z}_p(2)^{syn})$. By \cite[Theorem 5.2]{antieau2020beilinson} we have that $H^3(C;N\mathbb{Z}_p(2)^{syn})\to H^3((C/\sqrt{p});N\mathbb{Z}_p(2)^{syn})$ is an isomorphism, and the target vanishes by Lemma~\ref{lem: vanishing in positive characteristic for curves}. Thus the only group which can contribute to $NTC_1(C;\mathbb{Z}_p)$ is $H^1(C;N\mathbb{Z}_p(1)^{syn})\cong H^0(C\langle t\rangle;(\mathbb{G}_m)^\wedge_p)/H^0(C;(\mathbb{G}_m)^\wedge_p)$ and by using the inclusion map $B\mathbb{G}_m\to K(-)$ and the determinant map $K(-)\to B\mathbb{G}_m(-)$ we have that this group does contribute to $NTC_1$. Therefore \[K_1(C\langle t\rangle, (t);\mathbb{Z}_p)\cong H^0(C\langle t\rangle;(\mathbb{G}_m)^\wedge_p)/H^0(C;(\mathbb{G}_m)^\wedge_p)\] as desired.

    We now turn our attention to computing $K_2(C\langle t\rangle, (t);\mathbb{Z}_p)$. The syntomic cohomology groups which can contribute to this group are $H^0(C;N\mathbb{Z}_p(1)^{syn})$ and $H^2(C;N\mathbb{Z}_p(2)^{syn})$ when $C$ is affine. Since $C$ is smooth we have that $\mathbb{Z}_p(1)^{syn}\simeq (\mathbb{G}_m)^\wedge_p [-1]$ and so $H^0(C;N\mathbb{Z}_p(1)^{syn})\cong NH^{-1}(C;(\mathbb{G}_m)^\wedge_p)\cong 0$ since any $p^\infty$-torsion in $\mathbb{G}_m(C\langle t\rangle)$ must already be in $\mathbb{G}_m(C)$. The only possible group which can contribute to $K_2(C\langle t\rangle, (t);\mathbb{Z}_p)$ is therefore $H^2(C;N\mathbb{Z}_p(2)^{syn})$. This group cannot support or receive any nonzero differentials for degree reasons, so we get the desired isomorphism.
\end{proof}

\bibliographystyle{plainc}
\bibliography{bibliography}
\end{document}